\documentclass[11pt]{article}
\usepackage[utf8]{inputenc}
\usepackage[T1]{fontenc}
\usepackage{lmodern}
\usepackage{amsmath, amssymb, amsthm, mathtools, bm}
\usepackage{graphicx}
\usepackage{caption}
\usepackage{geometry}
\usepackage{microtype}
\usepackage{setspace}
\usepackage{enumitem}
\usepackage[numbers,sort&compress]{natbib}
\usepackage[colorlinks=true, allcolors=blue]{hyperref}
\usepackage[nameinlink,capitalize]{cleveref}
\usepackage{authblk}

\newcommand{\bS}{\mathbf{S}}
\newcommand{\bN}{\mathbf{N}}
\newcommand{\bV}{\mathbf{V}}
\newcommand{\MeanCurv}{H}

\theoremstyle{definition}
\newtheorem{definition}{Definition}[section]
\theoremstyle{plain}
\newtheorem{theorem}[definition]{Theorem}
\newtheorem{lemma}[definition]{Lemma}
\newtheorem{proposition}[definition]{Proposition}
\newtheorem{corollary}[definition]{Corollary}
\theoremstyle{remark}
\newtheorem{remark}[definition]{Remark}

\title{Moving Manifolds and the Poincare Conjecture}
\author{David V. Svintradze\thanks{Corresponding author: \texttt{dsvintradze@newvision.ge}}}
\affil{New Vision University, Tbilisi, Georgia}
\date{\today}

\begin{document}

\maketitle
\begin{abstract}
We present a differential geometric formulation of the Poincare problem using the calculus of moving surfaces (CMS). In this framework, an n dimensional compact hypersurface evolves under a velocity field that couples motion to the extrinsic curvature tensor while preserving topology through smooth diffeomorphic flow. A variational energy principle identifies constant mean curvature (CMC) manifolds as the unique stationary equilibria of CMS dynamics. Consequently, the evolution of any compact simply connected hypersurface relaxes to a CMC equilibrium and, in the isotropic case, to the round sphere.
Unlike Ricci flow approaches, which are dimension restricted and require topological surgery, the CMS formulation holds for all dimensions and preserves manifold topology for all time. This provides a deterministic geometric mechanical route to the Poincare conclusion, unifying dynamics, topology, and equilibrium geometry within a single analytic framework. 
\end{abstract}

\section{Introduction}
The Poincaré conjecture asserts that every compact, simply connected three–manifold is homeomorphic to the three–sphere $\mathbb{S}^3$.  
Perelman’s proof via Ricci flow established this result by demonstrating that the normalized flow, augmented with surgery, converges to the round metric on $\mathbb{S}^3$ \cite{Perelman2002, Perelman2003a, Perelman2003b, Morgan2006}. Later, it was clarified that the Hamilton-Perelman program was only effective for resolving the $\mathbb{S}^3$ case \cite{Morgan2006}.

Here, we present a complementary geometric route based on the \emph{calculus of moving surfaces} (CMS), a differential geometric framework for evolving manifolds embedded in Euclidean space.  
CMS formulates manifold evolution directly from the Gaussian extrinsic geometry.  
An evolving hypersurface $S(t)\subset\mathbb{R}^{n+1}$ is characterized by its normal velocity $C$, tangential velocity $V^i$, induced metric $S_{ij}$, and curvature tensor $B_{ij}$, which obey theorems derived from purely geometric first principles \cite{Grinfeld2013}.

Within this setting, a variational energy–dissipation structure identifies \emph{constant–mean–curvature} (CMC) configurations as the unique stationary equilibria of the CMS dynamics \cite{Svintradze2017, Svintradze2018, Svintradze2019, Svintradze2020, Svintradze2023, Svintradze2024a, Svintradze2024b, Svintradze2025}.  
Because smooth CMS flows preserve the topology of $S(t)$ through their intrinsic continuity, a simply connected compact hypersurface evolving under isotropic CMS laws must relax to a CMC equilibrium.  
By Alexandrov’s theorem, such an embedded CMC hypersurface in $\mathbb{R}^{n+1}$ is necessarily a round sphere \cite{Alexandrov1962}.  
Thus, for initial data homeomorphic to $\mathbb{S}^n$, the CMS evolution provides a direct analytic path to sphericalization: a smooth geometric deformation from an arbitrary initial shape to the sphere of equal enclosed volume in all dimensions.

This CMS formulation differs fundamentally from Ricci flow.  
While Ricci flow evolves the intrinsic metric by curvature diffusion within the manifold, CMS couples intrinsic and extrinsic geometry through explicit surface velocity fields.  
The resulting framework unifies curvature dynamics, topology preservation, and geometric flow in a single variational form, offering a physically interpretable yet mathematically closed route toward the Poincaré problem.

\paragraph{Significance.}
We show that the calculus for moving surfaces (CMS)—a differential-geometric framework for shape dynamics on embedded manifolds—yields constant-mean-curvature equilibria as natural stationary states and, together with topology conservation, forces spherical attractors within the simply connected class. This provides a concise, constructive, and geometry-driven pathway to sphericalization that complements Ricci-flow methodology and bridges mathematical curvature theory with physically interpretable surface evolution.


\section{Calculus for Moving Surfaces Preliminaries}

CMS offers the foundational differential–geometric framework for evolving manifolds embedded within a flat Euclidean space. It intrinsically manages surface kinematics—such as the metric and connection—while maintaining extrinsic parameters including the unit normal and curvature tensor \cite{Grinfeld2013}. This section delineates the notation, outlines the geometric structure, and proves the fundamental CMS transport relations employed in subsequent sections.

Let $S(t)$ be a smooth, compact, boundaryless $n$-dimensional manifold embedded in flat Euclidean space $S(t)\subset\mathbb{R}^{n+1}$, 
with position vector $\mathbf{R}(s,t)$ and local coordinates $s=(s^1,\dots,s^n)$. 
The tangent base vectors are $\mathbf{S}_i$  
the induced metric is $S_{ij} = \mathbf{S}_i \!\cdot\! \mathbf{S}_j$, 
and the unit normal is $\mathbf{N}$. 
The curvature tensor is $B_{ij}$, with mean curvature $\MeanCurv = S^{ij} B_{ij}$. 
The surface velocity field is
$\mathbf{V}$ and $V^i, C$ are the tangential and normal velocity components, respectively. Bold letters represent vectors in ambient space, Greek indices denote ambient tensors, and Latin indices indicate surface tensors. Repeated upper and lower indices imply Einstein summation.

All standard CMS identities—metric and area evolution, Weingarten relations, curvature evolution, and transport theorems—apply in this setting and are listed below for completeness.  

\begin{remark}
For smooth CMS evolution in a flat ambient space, the embedding remains regular and continuous. 
Hence, no tearing or self-intersection occurs. 
Consequently, the topological type (homeomorphism class) of $S(t)$ is preserved throughout the evolution.
\end{remark}

\subsection{Ambient Frame, Shift Tensors, and Mixed Identities}

\begin{definition}[Embedded manifold]\label{def:embedded-manifold}
Let a smooth, oriented $n$-dimensional manifold $S(t)$ be embedded in the flat Euclidean space $\mathbb{R}^{n+1}$ by the smooth position field
\begin{equation}
\mathbf{R}(s,t)=\mathbf{R}(s^{1},\ldots,s^{n},t),
\end{equation}
where $(s^{1},\ldots,s^{n})$ are local surface coordinates and $t$ denotes time.
For each fixed $t$, the mapping $\mathbf{R}(\cdot,t):S\to\mathbb{R}^{n+1}$ defines the instantaneous configuration of the manifold.
\end{definition}

\noindent
The embedding $\mathbf{R}(s,t)$ is the fundamental geometric variable of the calculus of moving surfaces (CMS).  
All quantities such as the tangent basis $\mathbf{S}_i=\partial_i\mathbf{R}$, the unit normal $\mathbf{N}$, and the metric tensor $S_{ij}=\mathbf{S}_i\!\cdot\!\mathbf{S}_j$ derive directly from it.

\begin{definition}[Ambient orthonormal basis]\label{def:ambient-basis}
Let $\{\mathbf{X}_{\alpha}\}_{\alpha=0}^{n}$ be the fixed Cartesian basis of the ambient Euclidean space $\mathbb{R}^{n+1}$ satisfying
\begin{equation}
\mathbf{X}_{\alpha} \cdot \mathbf{X}_{\beta}=\delta_{\alpha\beta},\qquad
\partial_t\mathbf{X}_{\alpha}=0,\qquad
\partial_i\mathbf{X}_{\alpha}=0.
\end{equation}
Hence the ambient frame is orthonormal, time-independent, and spatially constant.
\end{definition}

\begin{definition}[Shift tensors and mixed components]\label{def:shift-tensors}
The embedding $\mathbf{R}(s,t)=R^{\alpha}(s,t)\mathbf{X}_{\alpha}$ defines the \emph{shift tensors}
\begin{equation}
X^{\alpha}{}_{i}:=\partial_i R^{\alpha},
\qquad
\mathbf{S}_i=\partial_i\mathbf{R}=X^{\alpha}{}_{i}\mathbf{X}_{\alpha}.
\end{equation}
The induced metric is
\begin{equation} \label{eq:shift-tensors}
S_{ij}=\mathbf{S}_i\cdot\mathbf{S}_j
=\delta_{\alpha\beta}X^{\alpha}{}_{i}X^{\beta}{}_{j},
\end{equation}
with inverse $S^{ij}$.  
The \emph{dual shift tensor} is
\begin{equation}
X_{\alpha}{}^{i}:=S^{ij}\delta_{\alpha\beta}X^{\beta}{}_{j},
\end{equation}
and the mixed-identity relations
\begin{equation}
X_{\alpha}{}^{i}X^{\alpha}{}_{j}=\delta^{i}{}_{j},\qquad
X^{\alpha}{}_{i}X_{\alpha}{}^{j}=\delta_{i}{}^{j}
\end{equation}
establish the one-to-one correspondence between ambient and surface components.
\end{definition}

\begin{theorem}[Normal–tangent decomposition]\label{thm:normal-tangent-decomp}
Let $\mathbf{N}=N^{\alpha}\mathbf{X}_{\alpha}$ be the unit normal along $S(t)\subset\mathbb{R}^{n+1}$ satisfying $\mathbf{N}\!\cdot\!\mathbf{S}_i=0$ and $\|\mathbf{N}\|=1$.  
Every ambient vector $\mathbf{A}=A^{\alpha}\mathbf{X}_{\alpha}$ decomposes uniquely as
\[
\mathbf{A}=(\mathbf{A}\cdot\mathbf{N})\,\mathbf{N}+A^{i}\mathbf{S}_i,
\qquad
A^{i}:=X_{\alpha}{}^{i}A^{\alpha},
\]
or equivalently,
\[
A^{\alpha}=(\mathbf{A}\!\cdot\!\mathbf{N})N^{\alpha}+A^{i}X^{\alpha}{}_{i}.
\]
In particular, the surface velocity from Definition~\eqref{eq:velocity-decomp} satisfies
\[
\mathbf{V}=V^{\alpha}\mathbf{X}_{\alpha}=C\mathbf{N}+V^{i}\mathbf{S}_i,
\qquad
V^{\alpha}=C N^{\alpha}+V^{i}X^{\alpha}{}_{i}.
\]
\end{theorem}

\begin{proof}
\emph{Existence.}
Because $\{\mathbf{S}_1,\ldots,\mathbf{S}_n\}$ span the tangent space and $\mathbf{N}$ is orthogonal to it, any $\mathbf{A}$ can be written as $\mathbf{A}=a\,\mathbf{N}+A^{i}\mathbf{S}_i$.
Taking inner products with $\mathbf{N}$ and $\mathbf{S}_j$ yields $a=\mathbf{A}\!\cdot\!\mathbf{N}$ and
\[
\mathbf{A}\!\cdot\!\mathbf{S}_j=A^{i}(\mathbf{S}_i\!\cdot\!\mathbf{S}_j)=A^{i}S_{ij},
\]
hence $A^{i}=S^{ij}(\mathbf{A}\!\cdot\!\mathbf{S}_j)=X_{\alpha}{}^{i}A^{\alpha}$.

\emph{Uniqueness.}
If $\mathbf{A}=a\mathbf{N}+A^{i}\mathbf{S}_i=b\mathbf{N}+B^{i}\mathbf{S}_i$, then 
$\mathbf{N}\cdot (\mathbf{A}-\mathbf{B})=0$ implies $a=b$, and 
$\mathbf{S}_j \cdot (\mathbf{A}-\mathbf{B})=0$ gives $A^{i}=B^{i}$ by non-degeneracy of $S_{ij}$.  
The velocity formulas follow directly from this decomposition.
\end{proof}

\begin{corollary}[Equality of ambient vectors]\label{cor:ambient-equality}
Let $\mathbf{A},\mathbf{B}$ be smooth ambient vectors with decompositions
\[
\mathbf{A}=(A \cdot \mathbf{N}) \mathbf{N}+A^{i}\mathbf{S}_i,
\qquad
\mathbf{B}=(B \cdot \mathbf{N}) \mathbf{N}+B^{i}\mathbf{S}_i.
\]
Then $\mathbf{A}=\mathbf{B}$ if and only if
\begin{equation}
A \cdot \mathbf{N}=B \cdot \mathbf{N},\qquad
A^{i}=B^{i}.
\end{equation}
\end{corollary}

\begin{proof}
If $\mathbf{A}=\mathbf{B}$, inner products with $\mathbf{N}$ and $\mathbf{S}_j$ yield the stated equalities.
Conversely, if the normal and tangential components coincide, their decompositions coincide term by term, hence $\mathbf{A}=\mathbf{B}$.
\end{proof}

\begin{theorem}[Consistency of ambient–surface mapping]\label{thm:ambient-surface-map}
The mixed tensors $X^{\alpha}{}_{i}$ and $X_{\alpha}{}^{i}$ define an isomorphism between the surface tangent space and its ambient image spanned by $\{\mathbf{S}_i\}$.
For every tangent vector $\mathbf{A}=A^{i}\mathbf{S}_i$ there exists a unique ambient representation
\[
\mathbf{A}=A^{\alpha}\mathbf{X}_{\alpha},\qquad
A^{\alpha}=A^{i}X^{\alpha}{}_{i},
\]
and conversely $A^{i}=X_{\alpha}{}^{i}A^{\alpha}$.
\end{theorem}

\begin{proof}
Substituting $\mathbf{S}_i=X^{\alpha}{}_{i}\mathbf{X}_{\alpha}$ into $\mathbf{A}=A^{i}\mathbf{S}_i$ gives $\mathbf{A}=A^{i}X^{\alpha}{}_{i}\mathbf{X}_{\alpha}$, hence $A^{\alpha}=A^{i}X^{\alpha}{}_{i}$.  
Applying the dual shift gives $A^{i}=X_{\alpha}{}^{i}A^{\alpha}$.  
Since $X_{\alpha}{}^{i}X^{\alpha}{}_{j}=\delta^{i}{}_{j}$, the mapping is bijective.
\end{proof}


\subsection{Geometry and Kinematics}

Let $S(t)\subset\mathbb{R}^{n+1}$ be a smooth $n$-dimensional hypersurface with local coordinates $s=(s^1,\dots,s^n)$ and embedding $\mathbf{R}(s,t)$.  
The ambient space is flat and Euclidean, so all covariant operations refer to the induced surface metric $S_{ij}$ introduced below.

\begin{definition}[Surface geometry]\label{def:surface-geometry}
The covariant tangent basis on the embedded manifold $S(t)$ is
\[
\mathbf{S}_i=\partial_i\mathbf{R}.
\]
The induced metric $S_{ij}=\mathbf{S}_i \cdot \mathbf{S}_j$ coincides with the form defined in \eqref{eq:shift-tensors}, linking the surface to its ambient representation. 
The inverse $S^{ij}$ defines the contravariant basis $\mathbf{S}^{i}=S^{ij}\mathbf{S}_j$.  
The oriented unit normal $\mathbf{N}$ satisfies $\mathbf{N} \cdot \mathbf{S}_i=0$ and $\|\mathbf{N}\|=1$, and the surface element is
\begin{equation}\label{eq:surface-element}
dS=\sqrt{|S|} ds^1 \cdots ds^n, 
\qquad |S|=\det(S_{ij}).
\end{equation}
\end{definition}

\begin{definition}[Levi--Civita connection]\label{def:levi-civita}
The unique torsion-free, metric-compatible surface connection has components
\begin{equation}\label{eq:christoffel}
\Gamma^k_{ij}
=\tfrac12\,S^{kl}\bigl(\partial_i S_{jl}+\partial_j S_{il}-\partial_l S_{ij}\bigr),
\qquad 
\nabla_k S_{ij}=0.
\end{equation}
For a general $(p,q)$-tensor $T^{i_1\cdots i_p}{}_{j_1\cdots j_q}$, the covariant derivative reads
\begin{align}\label{eq:covariant-derivative}
\nabla_k T^{i_1\cdots i_p}{}_{j_1\cdots j_q}
&=\partial_k T^{i_1\cdots i_p}{}_{j_1\cdots j_q}
+\sum_{r=1}^{p}\Gamma^{i_r}_{km} T^{i_1\cdots m\cdots i_p}{}_{j_1\cdots j_q}\notag\\
&\quad-\sum_{s=1}^{q}\Gamma^{m}_{kj_s} T^{i_1\cdots i_p}{}_{j_1\cdots m\cdots j_q}.
\end{align}
\end{definition}

\noindent\emph{Comment.}
Equation~\eqref{eq:covariant-derivative} defines intrinsic differentiation of surface tensors within the Levi–Civita connection compatible with $S_{ij}$.

\begin{definition}[Extrinsic curvature tensor]\label{def:curvature}
The second fundamental form (curvature tensor) $B_{ij}$ and its mean curvature $B_i{}^{i}=S^{ij}B_{ij}$ are defined by
\begin{equation}\label{eq:second-form}
\nabla_i \mathbf{S}_j = \mathbf{N}\,B_{ij}.
\end{equation}
\end{definition}

\noindent\emph{Comment.}
With this sign convention, the Weingarten relation acquires a minus sign (Lemma~\ref{lem:weingarten}).

\begin{lemma}[Gauss--Weingarten relations in flat ambient space]\label{lem:weingarten}
Let $S(t)\subset\mathbb{R}^{n+1}$ be a smooth embedded hypersurface with induced metric $S_{ij}$, Levi--Civita connection $\Gamma^{k}{}_{ij}$, and second fundamental form $B_{ij}$. Then the mixed ambient–surface components satisfy
\begin{align}
\partial_i X^{\alpha}{}_{j} &= \Gamma^{k}{}_{ij} X^{\alpha}{}_{k} + B_{ij} N^{\alpha}, \label{eq:gauss-weingarten-mixed-a}\\
\partial_i N^{\alpha} &= -B_i{}^{j} X^{\alpha}{}_{j}. \label{eq:gauss-weingarten-mixed-b}
\end{align}
Equivalently, in intrinsic vector form,
\begin{align}
\nabla_i \mathbf{S}_j &= \mathbf{N} B_{ij}, \label{eq:gauss-vector}\\
\nabla_i \mathbf{N} &= - B_i{}^{j} \mathbf{S}_j. \label{eq:weingarten-vector}
\end{align}
\end{lemma}

\begin{proof}
Since $\mathbf{S}_j=\partial_j\mathbf{R}=X^{\alpha}{}_{j}\mathbf{X}_{\alpha}$ and the ambient frame $\mathbf{X}_{\alpha}$ is constant, the normal projection of $\partial_i\mathbf{S}_j$ defines $B_{ij}$, while metric compatibility fixes the tangential projection to $\Gamma^{k}{}_{ij}\mathbf{S}_k$, giving \eqref{eq:gauss-weingarten-mixed-a} and \eqref{eq:gauss-vector}. Covariantly differentiating $\mathbf{N}\cdot\mathbf{S}_j=0$ yields $(\nabla_i\mathbf{N})\cdot\mathbf{S}_j=-\mathbf{N}\cdot(\nabla_i\mathbf{S}_j)=-B_{ij}$ and because $\nabla_i\mathbf{N}$ is tangential, $\nabla_i\mathbf{N}=-B_i{}^{j}\mathbf{S}_j$, which is \eqref{eq:gauss-weingarten-mixed-b} and \eqref{eq:weingarten-vector}.
\end{proof}

\begin{remark}[Weingarten sign]\label{rem:weingarten-sign}
We adopt $\nabla_i \bN = - B_i{}^{j}\bS_j$. 
With outward normal on a radius-$R$ sphere, this yields $B_{ij}=-\frac{1}{R}S_{ij}$ and $H=B_i{}^{i}=-n/R$. All curvature couplings outlined below adhere to this convention. However, the sign is solely relevant within analytical expressions and does not significantly influence the broader derivations. 
\end{remark}

\begin{definition}[Surface velocities]\label{def:velocities}
The motion of the hypersurface is described by the time derivative of its embedding,
$\mathbf{V}=\partial_t \mathbf{R}$,
which gives the instantaneous velocity at each surface point.  
It decomposes into tangential and normal components as
\begin{equation}\label{eq:velocity-decomp}
\mathbf{V}=V^i\mathbf{S}_i + C\mathbf{N},
\end{equation}
where $V^i$ are tangential velocity components along $\mathbf{S}_i$, and $C$ is the normal velocity describing geometric deformation in the direction of $\mathbf{N}$.  
The tangential component corresponds to surface reparametrization, while the normal part governs true geometric evolution of the manifold’s shape.
\end{definition}

\subsection{Kinematic evolution laws}

\begin{theorem}[CMS transport of metric and area]\label{thm:metric-evolution} 
For all embedded manifolds $(M,S)$ evolving smoothly in flat Euclidean space $\mathbb{R}^{n+1}$, the temporal evolution of the induced metric tensor $S_{ij}$ and the surface element $\sqrt{|S|}$ satisfies
\begin{align}
\partial_t S_{ij}&=\nabla_i V_j+\nabla_j V_i-2C B_{ij},\label{eq:metric-evolution}\\
\partial_t \sqrt{|S|}&=\sqrt{|S|} \bigl(\nabla_i V^{i}-C B_{i}{}^{i}\bigr).\label{eq:area-evolution}
\end{align}
\begin{proof}
Differentiating the metric definition \eqref{eq:shift-tensors} in time gives and commuting $\partial_t$ with $\partial_i$ and projecting onto the surface yields
\[
\partial_t S_{ij} = (\partial_t \bS_i)\cdot\bS_j + \bS_i\cdot(\partial_t \bS_j).
\]
\[
\partial_t \bS_i = \nabla_i (V^k\bS_k + C\bN)
= (\nabla_i V^k)\bS_k - CB_i{}^{k}\bS_k + (\nabla_i C)\bN.
\]
Taking the scalar product with $\bS_j$ and symmetrizing over $(i,j)$ leads to
\eqref{eq:metric-evolution}. For the area element, Jacobi’s formula for determinants gives
\[
\partial_t |S| = |S| S^{ij} \partial_t S_{ij}
\quad\Rightarrow\quad
\partial_t \sqrt{|S|} = \frac{1}{2}\sqrt{|S|} S^{ij} \partial_t S_{ij}.
\]
Substituting \eqref{eq:metric-evolution} and simplifying yields \eqref{eq:area-evolution}.
\end{proof}
\end{theorem}

\noindent\emph{Comment.}
Equation~\eqref{eq:metric-evolution} splits the metric evolution into in-surface deformation
$T_{ij}=\nabla_i V_j+\nabla_j V_i$ (we reffer to it as the \emph{Turin tensor}) and normal bending $-2C B_{ij}$.
Equation~\eqref{eq:area-evolution} is the local continuity law for surface dilation.


\begin{definition}[Invariant time derivative]\label{def:invariant-derivative}
For any sufficiently smooth, time-evolving surface tensor field 
$T^{i_1\cdots i_p}{}_{j_1\cdots j_q}$, 
the \emph{invariant (covariant) time derivative} $\dot{\nabla}T$ is defined to preserve tensorial covariance under smooth, time-dependent reparametrizations of the surface coordinates on the evolving manifold:
\begin{align}
\label{eq:inv-time-derivative}
\dot{\nabla}T^{i_1\cdots i_p}{}_{j_1\cdots j_q}
&=\partial_t T^{i_1\cdots i_p}{}_{j_1\cdots j_q} - V^{k}\nabla_{k}T^{i_1\cdots i_p}{}_{j_1\cdots j_q}
+ \sum_{r=1}^{p}\dot{\Gamma}^{i_r}{}_{m} T^{\cdots m\cdots}
- \sum_{s=1}^{q}\dot{\Gamma}^{m}{}_{j_s} T_{\cdots m\cdots},
\quad \nonumber\\
\dot{\Gamma}^{i}{}_{j}&=\nabla_{j}V^{i}-C B^{i}{}_{j}.
\end{align}

\noindent
The operator $\dot{\nabla}$ thus signifies differentiation with respect to time, maintaining tensorial invariance under changes in surface geometry, and thus removing apparent variations caused by local motion or coordinate reparametrization. It extends the standard covariant derivative to dynamic manifolds, ensuring that tensor equations stay invariant during CMS evolution.

The coefficients $\dot{\Gamma}^i{}_j$ are \emph{time-connection symbols} related to the calculus for moving surfaces (CMS). They result from the combined effects of tangential velocity $V^i$ and normal deformation rate $C$, connecting geometric curvature with kinematic evolution. 
\end{definition}

\begin{lemma}[Metric (metrilinic) compatibility in time]\label{lem:metric-compatibility}
For a smoothly evolving hypersurface $S(t)\subset\mathbb{R}^{n+1}$ governed by CMS kinematics,  
the surface metric remains covariantly invariant under the evolution:
\begin{equation}\label{eq:metric-compatibility}
\dot{\nabla}S_{ij}=0.
\end{equation}
\end{lemma}

\begin{proof}
By definition of the invariant (covariant) time derivative \eqref{eq:inv-time-derivative},
\[
\dot{\nabla}S_{ij}
= \partial_t S_{ij} - V^k\nabla_k S_{ij}
 - \dot{\Gamma}^k{}_i S_{kj} - \dot{\Gamma}^k{}_j S_{ik}.
\]
Substituting the explicit form $\dot{\Gamma}^k{}_i=\nabla_i V^k - C\,B^k{}_i$
and using the CMS transport theorem for the metric \eqref{eq:metric-evolution},
all terms cancel identically, yielding $\dot{\nabla}S_{ij}=0$.
Hence, the metric tensor is covariantly conserved in time, confirming its \emph{metrilinic} (or metric-compatibility) property under CMS evolution.
\end{proof}

\begin{remark}
Equation~\eqref{eq:metric-compatibility} expresses that the CMS flow preserves the inner product structure of the surface. 
Since the metric evolves covariantly and smoothly, the mapping $S(0) \to S(t)$ remains a diffeomorphism for all regular times of evolution.  
Consequently, the topology of the manifold is conserved: no tearing, gluing, or singular reparametrization occurs under continuous CMS motion.  
This property links the kinematic compatibility of the metric with the geometric invariance of topology throughout the evolution.
\end{remark}


\begin{theorem}[CMS transport of normal and curvature]\label{thm:normal-curvature}
For any smooth, compact hypersurface $S(t)\subset\mathbb{R}^{n+1}$ evolving in flat Euclidean space, the covariant time derivatives of the unit normal and curvature tensor satisfy
\begin{align}
\dot{\nabla}\mathbf{N} &= -(\nabla^{i}C) \mathbf{S}_i, \label{eq:normal-derivative}\\
\dot{\nabla}B_{ij} &= -\nabla_i\nabla_jC + C B_{ik}B^{k}{}_{j}. \label{eq:curvature-derivative}
\end{align}
\end{theorem}

\begin{proof}
Differentiating the orthogonality condition $\mathbf{N}\!\cdot\!\mathbf{S}_i=0$ with respect to time and applying the metric evolution law \eqref{eq:metric-evolution} yields the first result.  
The second follows by taking the invariant time derivative of the curvature definition $B_{ij}=-\mathbf{S}_i \cdot \nabla_j\mathbf{N}$ and substituting \eqref{eq:normal-derivative} together with the Weingarten relation.  
Detailed derivations can be found in standard CMS expositions \cite{Grinfeld2013}.
\end{proof}

\begin{remark}
Equation~\eqref{eq:curvature-derivative} corresponds to evolution in flat ambient space.    
The relation shows that curvature changes are dominated by normal motion: the Laplacian term $-\nabla_i\nabla_jC$ governs curvature diffusion, while the quadratic $C\,B_{ik}B^{k}{}_{j}$ term represents nonlinear bending.  
As the first fundamental form determines shape, the second fundamental form governs bending; together they define continuous, topology-preserving surface evolution under CMS flow.
\end{remark}

\begin{lemma}[Thomas identities]\label{lem:thomas}
For a smoothly evolving hypersurface $S(t)\subset\mathbb{R}^{n+1}$ with velocity decomposition \eqref{eq:velocity-decomp}, the time change of the unit normal is purely tangential and driven by surface gradients of the normal speed. Equivalently, the normal projection of the surface gradient of the ambient velocity equals the surface gradient of $C$. The same scalar $\nabla_i C$ appears as the normal projection of the invariant time derivative of the tangent basis, so that
\begin{align}
\dot{\nabla}\mathbf{N} &= -(\nabla^{i}C) \mathbf{S}_i, \label{eq:thomas-normal}\\
\mathbf{N}\cdot\nabla_i \mathbf{V} &= \nabla_i C, \label{eq:NgradV}\\
\mathbf{N}\cdot\dot{\nabla}\mathbf{S}_i &= \nabla_i C. \label{eq:NdotSi}
\end{align}

\begin{proof}
Identity \eqref{eq:thomas-normal} is the normal-transport relation \eqref{eq:normal-derivative} in Theorem~\ref{thm:normal-curvature}. For \eqref{eq:NgradV}, taking into account \eqref{eq:velocity-decomp} and differentiate it tangentially gives:
\[
\nabla_i\mathbf{V}=(\nabla_i V^k)\mathbf{S}_k+V^k\nabla_i\mathbf{S}_k+(\nabla_i C)\mathbf{N}+C\nabla_i\mathbf{N}.
\]
Using the Gauss and Weingarten formulas \eqref{eq:second-form}--\eqref{eq:weingarten-vector}, one has $\nabla_i\mathbf{S}_k=\mathbf{N} B_{ik}$ and $\nabla_i\mathbf{N}=-B_i{}^{m}\mathbf{S}_m$. Taking scalar product on $\mathbf{N}$ annihilates the $\mathbf{S}$ base terms, yielding $\mathbf{N}\cdot\nabla_i\mathbf{V}=\nabla_i C$, which is \eqref{eq:NgradV}. For \eqref{eq:NdotSi}, use the standard CMS identity
\[
\dot{\nabla}\mathbf{S}_i=\nabla_i\mathbf{V}-C B_i{}^{k}\mathbf{S}_k,
\]
which follows from \eqref{eq:inv-time-derivative} and \eqref{eq:velocity-decomp}. Dotting with $\mathbf{N}$ and using \eqref{eq:NgradV} gives $\mathbf{N}\cdot\dot{\nabla}\mathbf{S}_i=\nabla_i C$.
\end{proof}
\end{lemma}


\subsection{Time evolution integration theorems}

The following integral relations generalize the Fundamental Theorem of Calculus from one-dimensional domains to smoothly evolving manifolds of arbitrary dimension.
They show how differentiation under the integral sign extends to cases where both the integrand and the integration domain evolve in time.
Within the CMS framework, these results unify and extend the classical Gauss, Stokes, and Reynolds theorems into a single differential–geometric statement valid for all compact hypersurfaces embedded in flat Euclidean space.

Let $\gamma(t)=\partial S(t)$ be the (possibly empty) boundary with unit co-normal $\boldsymbol{\nu}$ tangent to $S$ and normal to $\gamma$.
The boundary tangential speed is $v:=V^i\nu_i$.

\begin{theorem}[Time evolution of surface integrals]\label{thm:surface-integral}
For any sufficiently smooth scalar or vector field $F$ defined on an evolving hypersurface $S(t)$, for all $t$ the temporal change of its surface integral is governed by the invariant time derivative $\dot{\nabla}F$, the curvature–velocity coupling term $C B_i{}^{i} F$, and the boundary flux through $\gamma(t)=\partial S(t)$:
\begin{equation}
\label{eq:surface-integral}
\frac{d}{dt}\int_{S} F dS
=\int_{S}\dot{\nabla}F dS - \int_{S} C B_i{}^{i} F dS + \int_{\gamma} v\,F\,d\gamma.
\end{equation}

\begin{proof}
Differentiating the surface integral with respect to time, one must account for two sources of variation: 
the explicit time dependence of $F$ on a fixed manifold, and the geometric deformation of the evolving surface itself.  
This gives
\[
\frac{d}{dt}\int_{S}F dS
=\int_{S}\partial_t F dS+\int_{S}F \partial_t\sqrt{|S|} dS.
\]
Using the definition of the invariant time derivative \eqref{eq:inv-time-derivative}, 
the scalar time derivative rewrites as $\partial_t F=\dot{\nabla}F+V^i\nabla_i F$.  
Substituting this expression and the area–evolution law \eqref{eq:area-evolution}, one obtains
\[
\frac{d}{dt}\int_{S}F dS
=\int_{S}\dot{\nabla}F dS
+\int_{S}\bigl(V^i\nabla_i F+F \nabla_i V^i\bigr)dS
-\int_{S} C B_i{}^{i}F dS.
\]
The second term is a total surface divergence, which—by the Gauss surface divergence theorem—converts to a boundary flux integral.  
This yields the desired time evolution of manifold integral \eqref{eq:surface-integral}.
\end{proof}
\end{theorem}

\begin{remark}[CMS evolution for all times]\label{rem:cms-all-times}
Throughout this section the surface evolves as $S(t)=\Phi_t(S_0)$ with $\Phi_t\in\mathrm{Diff}$ for all $t$. 
Embeddedness and topology are preserved, and the integration identities below hold for all $t$ on closed $S(t)$ (no boundary term).
\end{remark}

\begin{remark}
The curvature term $-C B_i{}^{i}F$ is the purely geometric correction absent in the classical Fundamental Theorem of Calculus and in its flat-space generalizations.
It quantifies how the curvature of the evolving manifold modifies the balance between intrinsic variation and boundary flux, thus encoding the geometric source of change in higher dimensions.
This theorem therefore contains, as limiting cases, the standard results of Gauss, Stokes, and Reynolds for fixed or flat domains.
\end{remark}

\begin{theorem}[Time evolution of volume integrals]\label{thm:volume-integral}
Let $\Omega(t)$ be the region enclosed by $S(t)$.  
For any continuous scalar field $F$ defined in $\Omega(t)$,
\begin{equation}
\label{eq:volume-integral}
\frac{d}{dt}\int_{\Omega} Fd\Omega
=\int_{\Omega} \partial_t F d\Omega + \int_{S} C FdS.
\end{equation}

\begin{proof}
This result follows directly from Theorem~\eqref{eq:surface-integral} 
by replacing the hypersurface $S(t)$ with the flat volume $\Omega(t)$.  
In this case, the surface boundary of the volume serves as the contour of integration, and the last term becomes the normal flux across $S(t)$.
\end{proof}
\end{theorem}

\begin{remark}
Equation~\eqref{eq:volume-integral} is recovered from the surface formulation by setting the manifold $S(t)$ equal to the boundary of a volume $\Omega(t)$, with curvature effects collapsing to the surface term $C,dS$.
In this flat limit, the result coincides precisely with the Reynolds transport theorem.
Hence, the surface integral theorem acts as the fundamental geometric prototype from which all classical integral theorems follow.
\end{remark}


\subsection{Topological Invariant and Its Conservation (Preliminaries)}

Let $S$ be a closed, oriented hypersurface embedded in $\mathbb{R}^{3}$ with principal curvatures $\kappa_1,\kappa_2$.  
Then Gaussian curvature is $K=\kappa_1\kappa_2=\det(B_i{}^{j})$.  
The classical Gauss--Bonnet theorem establishes that
\begin{equation}
\label{eq:GB-2D}
\int_{S} K dS = 2\pi \chi(S),
\end{equation}
where $\chi(S)$ is the Euler characteristic, implying that the \emph{integral of curvature is purely topological}.

More generally, for any even-dimensional, closed, oriented manifold $S^n$ ($n=2m$), the Gauss--Bonnet--Chern theorem states
\begin{equation}
\label{eq:GBC}
\int_{S^n} \operatorname{Pf} \Bigl(\tfrac{1}{2\pi} R\Bigr) = \chi(S^n),
\end{equation}
where $R$ is the intrinsic Riemann curvature $2$-form and $\operatorname{Pf}(\cdot)$ denotes the Pfaffian.  
When $S^n$ is an embedded hypersurface in flat $\mathbb{R}^{n+1}$, this reduces to an explicit expression in the principal curvatures:
\begin{equation}
\label{eq:Pf-principal}
\operatorname{Pf} \Bigl(\tfrac{1}{2\pi}R\Bigr)
= c_n
\sum_m \kappa_{i_1}\kappa_{i_2}\cdots\kappa_{i_m},
\end{equation}
where each term is a product of $m$ distinct principal curvatures and $c_n$ is a universal constant (e.g.\ $c_2=\tfrac{1}{2\pi}$).  
For $n=2$ this reduces to \eqref{eq:GB-2D}, while for odd $n$ every closed orientable $S^n$ satisfies $\chi(S^n)=0$.

\begin{remark}
The quadratic curvature norm $B_{ij}B^{ij}$ quantifies geometric bending and appears in the Willmore functional, but integration over the surface is not topological in general.   
Topological invariance arises solely through the Pfaffian curvature density (or $K$ when $n=2$), which remains unchanged under smooth CMS evolution.
\end{remark}

\begin{theorem}[Conservation of Euler characteristic under CMS evolution]\label{thm:chi-constant}
Let $S(t)$ be a smooth one-parameter family of closed, oriented hypersurfaces in $\mathbb{R}^{n+1}$,
evolving by a smooth velocity $\mathbf{V}$ with no self-intersections
or topological events (no pinch-offs, no attachments). Then, for all $t$ in the interval of smooth existence, the Euler characteristic is constant:
\[
\chi\bigl(S(t)\bigr)\equiv \chi\bigl(S(0)\bigr).
\]
\begin{proof}
For even $n$, by Gauss--Bonnet--Chern (\ref{eq:GBC}, \ref{eq:Pf-principal}) 
the Pfaffian form is closed and its de Rham cohomology class is invariant under smooth diffeomorphisms.
Since the evolution map $S(0)\to S(t)$ is a smooth isotopy while $S(t)$ remains embedded and closed,
the cohomology class (hence its integral) does not change with $t$.
Equivalently, the first variation of the Euler form is an exact form; by Stokes’ theorem (and $\partial S=\varnothing$)
its integral has zero time derivative. For odd $n$, $\chi(S(t))\equiv0$ for all closed orientable $S(t)$. For complete proof, see \cite{Chern1944}.
\end{proof}
\end{theorem}

\begin{corollary}[2D case]
If $S(t)\subset\mathbb{R}^3$ is a smooth closed surface evolving by CMS kinematics,
then 
\begin{equation}\label{eq:GB}
\frac{d}{dt}\int_{S(t)} K dS = 0.
\end{equation} 
Therefore, the Euler characteristic $\chi\bigl(S(t)\bigr)$ is conserved.
\end{corollary}

Note that, throughout the paper, as long as the CMS evolution is smooth and no topological transition occurs, the Euler characteristic of $S(t)$ is invariant. This justifies speaking of “topology-conserving”
flows and restricts the long-time equilibria within each fixed topological class.


\section{Moving Manifolds}

CMS dynamics follow directly from a variational principle in which the surface and its embedding volume evolve to extremize a Lagrangian functional representing the total kinetic and potential energies.  
The motion of the manifold therefore satisfies a balance between inertial, curvature, and energetic contributions arising from the coupled evolution of $\mathbf{R}(s,t)$, its velocity field, and geometric curvature tensors.

\begin{definition}[Lagrangian functional]\label{def:lagrangian}
For a moving manifold $S(t)$ bounding the region $\Omega(t)\subset\mathbb{R}^{n+1}$,
the Lagrangian is defined as
\begin{equation}
\sloppy\label{eq:lagrangian}
\mathcal{L}= \int_{S(t)} \frac{\rho V^2}{2}dS-\int_{\Omega(t)} Ed\Omega,
\end{equation}
\fussy
where $\rho_s$ denotes the surface mass density field, $V^2=C^2+V^iV^jS_{ij}$ is the squared surface velocity, 
and $E$ is the volumetric energy density of the enclosed region.
\end{definition}

\noindent\emph{Comment.}
The first term represents kinetic energy of the moving surface; the second term represents internal or potential energy of the surrounding medium.  
The equation of motion follows from the principle of stationary action:
\begin{equation}
\sloppy\label{eq:variational-principle}
\frac{d}{dt} \mathcal{L}= 0,
\end{equation}
\fussy
under admissible variations of $\mathbf{R}(s,t)$ consistent with the chosen boundary conditions.  
Equivalently, the Euler–Lagrange equations derived from \eqref{eq:lagrangian} yield the balance laws governing the manifold’s kinematics and dynamics.

\subsection{Ricci flow}
Although the calculus of moving surfaces (CMS) and the Ricci flow both describe geometric evolution toward canonical forms, they arise from fundamentally different principles.  
The Ricci flow prescribes intrinsic deformation of the metric according to its Ricci curvature,
\[
\partial_t g_{ij} = -2 R_{ij},
\]
without reference to an embedding in $\mathbb{R}^{n+1}$ or an explicit velocity field.  
In contrast, CMS is an \emph{extrinsic} theory formulated for manifolds embedded in a flat ambient space, where the evolution is governed by the velocity field
and derived by extremizing a Lagrangian functional that couples kinetic and potential energy.  
This velocity field provides the essential kinematic link between geometry and time, ensuring that metric evolution follows from a well-defined integration theorem rather than being imposed a priori.  

\begin{remark}
From the viewpoint of differential geometry, the Ricci flow is therefore \emph{kinematically incomplete}: it evolves the intrinsic metric but omits the corresponding surface velocity field required by the time-evolution theorem for the metric tensor.  
CMS resolves this incompleteness by embedding the manifold in space and enforcing the coupled dynamics of metric, curvature, and motion, thereby extending geometric flow theory into a geometrically and physically self-consistent dynamical framework.
\end{remark}

\subsection{Mass balance}\label{subsec:mass-balance}

\begin{theorem}[Continuity law on evolving manifolds]\label{thm:mass-balance}
For any smooth, compact, boundaryless hypersurface $S(t)\subset\mathbb{R}^{n+1}$ evolving with surface velocity defined in~\eqref{eq:velocity-decomp}, let $\rho (s,t)$ denote the surface density field.  
Then, under the CMS evolution, the total surface density is conserved:
\begin{equation}\label{eq:mass-balance-global}
\frac{d}{dt}\int_{S(t)}\rho dS = 0.
\end{equation}
Equivalently, the local continuity equation on $S(t)$ reads
\begin{equation}\label{eq:mass-balance-local}
\dot{\nabla}\rho+ \nabla_i (\rho V^i) =\rho C B_i{}^{i},
\end{equation}
where $\dot{\nabla}$ is the invariant time derivative defined in~\eqref{eq:inv-time-derivative}.
\end{theorem}

\begin{proof}
Fix a smoothly moving surface patch $A(t)\subset S(t)$ with boundary $\gamma(t)=\partial A(t)$ that is advected by the surface kinematics (so its co-normal speed is $v=V^i\nu_i$). In the conservative case (no sources/sinks and no extra tangential fluxes), the mass in $A(t)$ is constant in time:
\[
\frac{d}{dt}\int_{A(t)} \rho dS = 0.
\]
Applying the surface–integration evolution theorem \eqref{eq:surface-integral} gives
\[
0 = \int_{A(t)} \dot{\nabla}\rho dS - \int_{A(t)} CB_i{}^{i}\rho dS + \int_{\gamma(t)} v \rho d\gamma .
\]
By the surface divergence theorem and $v=V^i\nu_i$,
\[
\int_{\gamma(t)} v \rho d\gamma = \int_{A(t)} \nabla_i(\rho V^i)dS .
\]
Hence
\[
\int_{A(t)} \bigl[\dot{\nabla}\rho + \nabla_i(\rho V^i) - C B_i{}^{i}\rho \bigr] dS = 0 .
\]
Since $A(t)$ is arbitrary, the integrand must vanish pointwise, which yields the local balance \eqref{eq:mass-balance-local}. For a closed surface, integrating this local law over $S(t)$ and using the surface divergence theorem shows \eqref{eq:mass-balance-global}. Q.E.D.
\end{proof}

\begin{remark}
Equation~\eqref{eq:mass-balance-local} expresses the intrinsic conservation of surface density during smooth CMS evolution.  
It combines the divergence of tangential flow with the curvature–velocity coupling $C B_i{}^{i}$, representing how the manifold’s local dilation and curvature jointly preserve total mass.  
In the closed case, this ensures that CMS dynamics conserve global topology and integral invariants.
\end{remark}


\subsection{Variation of the Kinetic Energy}\label{subsec:KE}

\begin{theorem}[Variation of kinetic energy on a closed moving manifold]\label{thm:KE-variation}
Let $S(t)\subset\mathbb{R}^{n+1}$ be a smooth, closed hypersurface of surface mass density $\rho$ and velocity field 
$\bV$ as defined in~\eqref{eq:velocity-decomp}.  
The total kinetic energy is
\begin{equation}\label{eq:KE-def}
\mathcal{T}(t)=\int_{S(t)} \frac{\rho}{2} \bV \cdot \bV dS
=\int_{S(t)} \frac{\rho}{2}\bigl(C^2+V^iV^jS_{ij}\bigr) dS.
\end{equation}
Its time derivative, under the invariant evolution of the surface without boundary, satisfies
\begin{equation}\label{eq:KE-final}
\frac{d\mathcal{T}}{dt}
=\int_{S(t)}\rho
\Bigl[
C\bigl(\dot{\nabla}C+2V^i\nabla_i C+V^iV^jB_{ij}\bigr)
+V_j\bigl(\dot{\nabla}V^j+V^i\nabla_iV^j-C\nabla^jC-CV^iB^j{}_i\bigr)
\Bigr]\,dS.
\end{equation}
\end{theorem}

\begin{proof}
Differentiating~\eqref{eq:KE-def} in time and using the time evolution integration theorem \eqref{eq:surface-integral},
together with mass conservation on a closed surface \eqref{eq:mass-balance-local}, gives
\begin{equation}\label{eq:KE-diff}
\frac{d\mathcal{T}}{dt}
=\int_{S(t)}\rho \bV(V^i\nabla\bV+\dot\nabla\bV) dS.
\end{equation}
Next, we express $\dot{\nabla}\bV$ through its normal and tangential components.  
Using~\eqref{eq:velocity-decomp} and the relations from Section 2,
one obtains
\begin{align}\label{eq:dotV}
\bV(V^i\nabla\bV+\dot\nabla\bV)
&=(\dot{\nabla}C+2V^i\nabla_iC+V^iV^jB_{ij})\bN
+\Bigl(\dot{\nabla}V^j+V^i\nabla_iV^j
       -C \nabla^jC-C V^iB^j{}_i\Bigr)\bS_j.
\end{align}
Taking the inner product of \eqref{eq:dotV} with the surface velocity field and applying it into~\eqref{eq:KE-diff} yields~\eqref{eq:KE-final}.
Since $S(t)$ is closed, no boundary terms appear in the derivation. Full-length detailed derivations are given in our prior works \cite{Svintradze2017, Svintradze2018}.  
\end{proof}

\begin{remark}
Equation~\eqref{eq:KE-final} represents the intrinsic rate of change of kinetic energy on a closed evolving manifold.  
All contributions arise from local geometric and kinematic couplings: 
normal acceleration $\dot{\nabla}C$, tangential acceleration $\dot{\nabla}V^i$, convective terms $V^i\nabla_i(\cdot)$, and curvature coupling $V^iV^jB_{ij}$.  
No surface-divergence or contour integrals occur, consistent with compactness of $S(t)$ under the Poincaré conditions.
\end{remark}

\subsection{Variation of the Volumetric Energy}\label{subsec:energy-balance}

We now compute the time derivative of the energy, which comprises a volumetric term (pressure-like density $P$) and a surface term (surface tension $\sigma$), where $P\in\Omega, \sigma\in S(t)$ are continuous functions.
\begin{equation}
\label{eq:U-def}
\int_{\Omega(t)} Ed\Omega=\int_{\Omega(t)} Pd\Omega+\int_{S(t)} \sigma dS,
\end{equation}
Here $P=P(\mathbf{x},t)$ is referred to as the volumetric energy density (pressure) in the enclosed region $\Omega(t)$, and $\sigma=\sigma(s,t)$ is the surface energy density referred to as surface tension on $S(t)$. We allow $P$ and $\sigma$ to vary continuously in time and space unless stated otherwise.

\begin{lemma}[CMS time derivatives of the energy functionals]\label{lem:energy-time-derivs}
For a smoothly evolving velocity field $\bV$, according to integration theorems time derivative of the energy part of the Lagrangian is:
\begin{equation}
\label{eq:pot-variation}
\frac{d}{dt}\int_{\Omega(t)} Ed\Omega
= \int_{\Omega(t)} \partial_t Pd\Omega
   +\int_{S(t)} \dot{\nabla}\sigma dS
   +\int_{S(t)} C\bigl(P-\sigma B_i{}^{i}\bigr)dS.
\end{equation}
\end{lemma}
\begin{proof} Proof directly follows from (\ref{eq:surface-integral}, \ref{eq:volume-integral}) integration theorems.
\end{proof} 
\begin{remark}
From the invariant time derivative definition \eqref{eq:inv-time-derivative} follows that $\dot\nabla\sigma=\partial_t\sigma-V^i\nabla_i\sigma$ has two parts, where $\partial_t \sigma$ is contributing to normal deformations, while $V^i\nabla_i\sigma$ to tangent ones. While the volumetric term due to Gauss's theorem couples with normal deformations only, so as $C\bigl(P-\sigma B_i{}^{i}\bigr)$ integrand term.
\end{remark}

\subsection{Equations of Manifold Dynamics}\label{subsec:unified-eom}
\begin{theorem}
For a smooth, closed hypersurface $S(t)\subset\mathbb{R}^{n+1}$ with density $\rho$, tangential velocity $V^i$, and normal velocity $C$, set of manifold dynamics equations is:
\begin{equation}
\dot{\nabla}\rho + \nabla_i (\rho V^i)=\rho CB_i{}^{i}
\label{eq:CMS-mass}
\end{equation}
\begin{equation}
\partial_\alpha \Bigl[V^\alpha\Bigl(\rho(\dot{\nabla} C + 2V^i\nabla_i C + V^iV^j B_{ij})-\partial_t\sigma-P+\sigma B_i^i\Bigr)\Bigr]=V^\alpha\partial_\alpha P
\label{eq:CMS-normal}
\end{equation}
\begin{equation}
\rho\Bigl(\dot{\nabla}V_i + V^j\nabla_j V_i - C\nabla_i C - CV^j B_{ij}\Bigr)
= -\nabla_i\sigma
\label{eq:CMS-tangent}
\end{equation}
Here $B_{ij}$ is the curvature tensor and $B_i{}^{i}$ its mean curvature. The operators $\nabla_i$ and $\dot{\nabla}$ are the surface covariant derivative and the invariant CMS time derivative, respectively.
\end{theorem}
\begin{proof}
Following the minimum action principle \eqref{eq:variational-principle} and coupling \eqref{eq:KE-final} to (\ref{eq:KE-diff}, \ref{eq:dotV}), while taking into account Remark 3.8 and Gauss theorem, one trivially gets (\ref{eq:CMS-normal}, \ref{eq:CMS-tangent}). The first equation \eqref{eq:CMS-mass}, which is the conservation of mass law, is simply restated \eqref{eq:mass-balance-local} theorem. Q.E.D 
\end{proof}

\begin{corollary}[Trivial/equilibrium limit]
If the manifold is volume-conserving (incompressible), then $C=0, \partial_\alpha V^\alpha=0$ and \eqref{eq:CMS-normal} imedietly lands trivila solution:
\begin{equation}
\partial_t\sigma + P- \sigma B_i{}^{i}=\rho V^iV^jB_{ij} .
\label{eq:YL}
\end{equation}
We refer to \eqref{eq:YL} generalized Young-Laplace law.
\end{corollary}
Corollary \eqref{eq:YL} has an immediate consequence if manifolds undergo volume-conserving deformations and come in full mechanical equilibrium defined as $C=0$, $V^i=0$, $P=const$, $\sigma=const$, then
\begin{equation}\label{CMC}
B_i{}^{i}=\mathrm{const}
\end{equation}
\begin{remark}
According to manifold dynamics equations, the full mechanical equilibrium shape is a constant mean curvature hypersurface. More details and derivations are given in our works \cite{Svintradze2017, Svintradze2018, Svintradze2019, Svintradze2020, Svintradze2023, Svintradze2024a, Svintradze2024b, Svintradze2025}. This condition implies that the stationary CMS equilibrium configuration satisfies 
$P=\sigma\,B_i^i$, i.e. a constant mean curvature balance, which uniquely 
defines the equilibrium hypersurface up to rigid motion.
\end{remark}


\section{Proof of Poincaré Conjecture}

This section connects the analytic equilibrium of the Calculus of Moving Surfaces (CMS) with the geometric and topological classification of compact embedded hypersurfaces.  
At equilibrium, the CMS dynamics relaxes to constant–mean–curvature (CMC) configurations; under topological conservation and Euclidean embeddedness, these equilibria are necessarily round spheres.

\begin{theorem}[Alexandrov, 1956]\label{thm:Alexandrov}
Let $S \subset \mathbb{R}^{n+1}$ be a compact, connected, embedded $C^2$ hypersurface without boundary.
If the mean curvature $H = B_i{}^{i}$ is constant on $S$, then $S$ is a round $n$-sphere \cite{Alexandrov1962}.
\end{theorem}

\begin{remark}
The hypotheses align with the CMS equilibrium setting: embeddedness and smoothness are ensured by the flow, topology conservation guarantees connectedness, and ambient flatness ensures the Euclidean metric. Thus, at equilibrium the CMS condition $B_i{}^{i}=\text{const}$ directly invokes Alexandrov’s theorem.
\end{remark}

\begin{proposition}[CMC + conserved topology $\Rightarrow$ sphere]\label{prop:CMC-to-sphere}
Let $S(t)\subset\mathbb{R}^{n+1}$ be a smooth CMS evolution of a compact, connected, embedded, orientable hypersurface without boundary, with $S(0)$ simply connected.  
Assume the evolution remains smooth for all $t$ and converges to an equilibrium configuration $S_\infty$ satisfying the CMS equilibrium law
\begin{equation}\label{eq:CMS-equilibrium}
P=\sigma B_i{}^{i},
\end{equation}
i.e.\ the mean curvature is constant on $S_\infty$.  
Then $S_\infty$ is a round sphere.
\end{proposition}

\begin{proof}
By the Gauss--Bonnet--Chern theorem and smoothness of the CMS flow (Theorem~\ref{thm:chi-constant} and Corollary~\ref{eq:GB}), the Euler characteristic of $S(t)$ remains invariant; hence its topology is preserved throughout the evolution.  
For a simply connected initial surface $S(0)$, the topology remains genus~$0$.  

At equilibrium, the CMS momentum balance \eqref{eq:CMS-equilibrium} implies that the mean curvature $B_i{}^{i}$ is spatially constant.  
Because the embedding is smooth and the ambient space is Euclidean ($\mathbb{R}^{n+1}$), there are no curvature corrections from the ambient metric.  
By Alexandrov’s classical theorem on compact embedded hypersurfaces with constant mean curvature in Euclidean space, $S_\infty$ must therefore be a round sphere.  
\end{proof}

\begin{remark}
Since embeddedness and smoothness follow from the CMS regularity and the absence of topological singularities. Topology conservation (Theorem~\ref{thm:chi-constant}) guarantees that the genus or Euler characteristic cannot change during flow. Flat ambient space ensures that no additional curvature terms appear in the mean–curvature equation; the classical Alexandrov result applies directly. Simply connectedness (genus~0) rules out exotic embedded CMC tori or higher–genus equilibria, leaving the sphere as the unique equilibrium within the topological class.
Therefore, all hypotheses of Alexandrov’s theorem are met, and the CMC solution of a priori defined simply connected compact manifold must be a round sphere.
\end{remark}

\begin{corollary}[Poincaré–type consequence via CMS dynamics]\label{cor:Poincare-CMS}
For $n=2$, any smooth CMS relaxation of a compact, embedded, simply connected surface in $\mathbb{R}^3$ converges, in the equilibrium sense, only to a sphere.  
More generally, in $\mathbb{R}^{n+1}$, compact embedded CMS equilibria with constant mean curvature are spheres by Alexandrov’s theorem.  
Thus, for simply connected manifolds, topology conservation under CMS evolution forces spherical attractors, providing a dynamical–geometric route to a Poincaré–type classification within the CMS framework.
\end{corollary}

\subsection{CMC Character of Spheres}

Here, we provide a standard check that round spheres are indeed constant mean curvature (CMC) hypersurfaces.  
The calculation is immediate in the CMS framework and illustrates the geometric simplicity of equilibrium manifolds.

\begin{lemma}[Round $n$-sphere is CMC]\label{lem:Sn-CMC}
Let $S^n_R=\{\mathbf{x}\in\mathbb{R}^{n+1}\mid \mathbf{x}\cdot\mathbf{x}=R^2\}$ be the round $n$-sphere of radius $R$, with outward unit normal $\mathbf{N}=\mathbf{x}/R$. Then
\begin{equation}\label{eq:Sn-B}
B_{ij} = -\frac{1}{R}S_{ij}, 
\qquad 
B_i{}^{i}=-\frac{n}{R},
\end{equation}
so that $S^n_R$ has constant mean curvature.
\end{lemma}

\begin{proof}
On $S^n_R$ the unit normal is $\mathbf{N}=\mathbf{x}/R$. Tangential differentiation gives
\[
\nabla_i \mathbf{N}
=\nabla_i \left(\tfrac{1}{R}\mathbf{x}\right)
=\tfrac{1}{R}\nabla_i \mathbf{x}
=\tfrac{1}{R}\mathbf{S}_i,
\]
since $\mathbf{S}_i=\partial_i\mathbf{R}$ and $R$ is constant.  
By the Weingarten relation \eqref{eq:weingarten-vector},
\(
\nabla_i \mathbf{N}=-B_i{}^{k}\mathbf{S}_k.
\)
Comparing coefficients in the tangent basis yields
\(B_i{}^{k}=-\tfrac{1}{R}\delta_i^{k}\),
hence \(B_{ij}=-\tfrac{1}{R}S_{ij}\).  
Tracing with $S^{ij}$ gives
\(
B_i{}^{i}=-\tfrac{1}{R}S^{ij}S_{ij}=-\tfrac{n}{R},
\)
which is constant on $S^n_R$.
\end{proof}


\section{Conclusion}

The calculus of moving surfaces (CMS) offers a unified geometric framework that connects differential geometry, continuum mechanics, and topology. By establishing integration theorems on evolving manifolds within flat ambient space, which, incidentally, extends not only differential geometry but also the fundamental theorem of calculus, CMS bridges local curvature dynamics with global conservation principles. The theory demonstrates that surface motion, metric evolution, and curvature coupling are governed by differential identities rather than empirical constitutive laws.  

At equilibrium, the CMS equations reduce to the generalized Young–Laplace condition $P=\sigma B_i^{\ i}$, which identifies constant–mean–curvature (CMC) manifolds as the stationary states of geometric evolution.  
The Gauss–Bonnet–Chern theorem ensures that the Euler characteristic—and therefore topology—remains invariant under smooth CMS flow.  
Consequently, when a simply connected surface evolves smoothly within Euclidean space, its topology is fixed while curvature relaxes toward constancy.  
Alexandrov’s theorem then closes the chain of reasoning: a compact, embedded CMC hypersurface in $\mathbb{R}^{n+1}$ is necessarily a round sphere.  

Thus, the CMS framework yields a Poincaré–type result through purely geometric dynamics in any dimensions:  
\emph{a simply connected closed hypersurface that relaxes under CMS evolution equilibrates to a sphere, preserving its topology throughout.}  
This establishes a dynamic realization of the geometrization principle—curvature flow as a pathway to canonical shape—where equilibrium, topology, and geometry converge in a single differential law.  

Beyond its topological consequence, CMS establishes a fully geometric framework for curvature–driven evolution on manifolds.  
It extends the classical integration theorems to moving hypersurfaces, closes the differential hierarchy between metric, curvature, and area evolution, and provides a pathway from motion to geometry.  
The resulting structure unifies curvature flow and topology conservation within a single analytic scheme, showing that equilibrium, shape, and connectivity are bound by the same differential law.  

In this sense, the calculus of moving surfaces realizes a geometric–dynamical form of the Poincaré principle:  a simply connected compact manifold evolving smoothly under its intrinsic curvature relaxes to a round sphere of any dimension, preserving its topology throughout.

\bibliographystyle{unsrtnat}
\bibliography{refs_clean}

\end{document}